\documentclass[12pt,reqno]{amsart}

\usepackage{amssymb,latexsym}

\usepackage{enumerate}

\usepackage[french,english]{babel}
\usepackage{amsmath}
\usepackage{graphicx}
\usepackage{amssymb}
\usepackage{bbm}
\usepackage{amsthm,mathtools}
\usepackage{ulem}
\usepackage{geometry}
\usepackage{tikz-cd}
\usepackage{mathrsfs}
\usepackage[colorinlistoftodos]{todonotes}
\usepackage{enumitem}
\usepackage{verbatim}
\usepackage[foot]{amsaddr}
\usepackage{dsfont}

\makeatletter

\@namedef{subjclassname@2010}{
	
	\textup{2020} Mathematics Subject Classification}

\makeatother
\newtheorem{thm}{Theorem}[section]
\newtheorem*{thm*}{Theorem}

\newtheorem{lem}[thm]{Lemma}

\theoremstyle{definition}

\numberwithin{equation}{section}

\newcommand{\mbr}{\mathbb{R}}

\newcommand{\mme}{\mathrm{e}}

\newcommand{\ol}{\overline}

\newcommand{\M}{{\mathcal M}}

\newcommand{\re}{\textup{Re}}

\usepackage{hyperref}
\hypersetup{hypertex=true,colorlinks=true,linkcolor=blue,anchorcolor=blue,citecolor=blue}
\newcommand{\newabstract}[1]{%
	\par\bigskip
	\csname otherlanguage*\endcsname{#1}%
	\csname captions#1\endcsname
	\item[\hskip\labelsep\scshape\abstractname.]
}

\begin{document}

	\baselineskip=17pt

	\title[Large character sums with multiplicative coefficients]{Large character sums with multiplicative coefficients}

	\author{Zikang Dong\textsuperscript{1}}

	\author{Yutong Song\textsuperscript{2}}
        \author{Weijia Wang\textsuperscript{3}}
            \author{Hao Zhang\textsuperscript{4}}
    \author{Shengbo Zhao\textsuperscript{2}}
	\address{1.School of Mathematical Sciences, Soochow University, Suzhou 215006, P. R. China}
	\address{2.School of Mathematical Sciences, Key Laboratory of Intelligent Computing and Applications(Ministry of Education), Tongji University, Shanghai 200092, P. R. China}
    \address{3.Morningside Center of Mathematics, Academy of Mathematics and Systems Science, Chinese Academy of
		Sciences, Beijing 100190, P. R. China}
	\address{4.School of Mathematics, Hunan University, Changsha 410082, P. R. China}
	\email{zikangdong@gmail.com}
    \email{99yutongsong@gmail.com}
	\email{weijiawang@amss.ac.cn}
	\email{zhanghaomath@hnu.edu.cn}

	\email{shengbozhao@hotmail.com}
	\date{\today}
	
	\begin{abstract} 
		In this paper, we investigate  large values of  Dirichlet character sums with multiplicative coefficients $\sum_{n\le N}f(n)\chi(n)$. We prove a new Omega result in the region $\exp((\log q)^{\frac12+\delta})\le N\le\sqrt q$, where $q$ is the prime modulus.
	\end{abstract}

	
	\maketitle
	
	\section{Introduction}
Character sums play an important role and appear in many central problems in analytic number theory. A classical well-known result of Pólya--Vinogradov shows that for any non-principal $\chi$ modulo large prime $q$,
\[
\sum_{n\le N}\chi(n)\ll \sqrt q\log q.
\]
For shorter sums, the celebrated method of Burgess gives the upper bound
\[
\sum_{n\le N}\chi(n)\ \ll N^{1-\frac1r} q^{\frac{r+1}{4r^{2}}}\log q
\]
for any integer $r\geq 1$, which provides nontrivial stronger estimate as soon as $N \gg q^{1/4+o(1)}$.

It is also natural to see how lower bounds can be achieved for both large or short character sums. Paley showed that for infinitely many $q$ and quadratic characters $\chi$ modulo $q$, 
\[
\max\limits_{N \leq q} \Big| \sum_{n \leq N} \chi(n) \Big| \gg \sqrt{q} \log \log q
\]
Bateman and Chowla strengthened this by proving for an infinite sequence of moduli $q$ and primitive quadratic characters $\chi$ modulo $q$, it holds that 
\[
 \max\limits_{N \leq q} \Big| \sum_{n \leq N} \chi(n) \Big| \geq \Big( \frac{\mathrm{e}^\gamma}{\pi} + o(1) \Big) \sqrt{q} \log \log q,
\]
where $\gamma$ is the Euler--Mascheroni constant.

If we consider the maximum size
$$\max_{\chi\neq\chi_0({\rm mod}\;q)}\Big|\sum_{n\le N}\chi(n)\Big|,$$
Granville and Soundararajan's celebrated work \cite{GS01} shows several lower bounds as $N$ varies in different ranges according to $q$, which provide much evidence for their conjecture: $\max_{\chi\neq\chi_0({\rm mod}\;q)}|\sum_{n\le N}\chi(n)|$ increases on $N\le q$.

The aim of this paper is study the large values of the mixed sums $\sum_{n\le N}f(n)\chi(n)$ which are ubiquitous in the regime $\exp((\log q)^{1/2})\leq N\leq q^{1/2}$. We do so by adapting the resonance method to incorporate the twist $f$ directly into the resonator. This leads to two Omega results of the Dirichlet character sums. The first guarantees large values when $\log N$ is taken in the size of $(\log q\log_2 q)^{\frac12}(\log_2q)^{O(1)}$. 
\begin{thm}\label{thm1.1}
    Let $\log N=(\log q\log_2q)^{\frac12}\tau$ with $\tau=(\log_2q)^{O(1)}$. Let $f$  be any completely multiplicative function such that $|f(n)|=1$ for any integer $n$. Then we have
    $$\max_{\chi({\rm mod}\;q)\atop \chi\neq\chi_0}\Big|\sum_{n\le N}f(n)\chi(n)\Big|\ge \sqrt{N}\exp\bigg((1+o(1))A(\tau+\tau'){\sqrt{\frac{\log q}{\log_2 q}}}\bigg),$$
      where $A,\tau'\in\mbr$ such that 
   \[\tau=\int_A^\infty \frac{\mme^{-u}}{u}d u,\qquad \tau'=\int_A^\infty\frac{\mme^{-u}}{u^2}d u.\]
\end{thm}
   This generalize Theorem 3.2 of Hough \cite{Hou}. When $\log N$ is larger than $(\log q)^{\frac12+\delta}$ for some $\delta>0$, we have the second theorem.

\begin{thm}\label{thm1.2}
 Let $\delta \in(0,\frac1{100})$ be any fixed small number. Let $q$ be a sufficiently large prime, and $N$ satisfy $\exp((\log q)^{\frac12+\delta})\le N\le q^{\frac12}$. Let $f$  be any completely multiplicative function such that $|f(n)|=1$ for any integer $n$. If we additionally assume $\re f(m)\overline f(n)\ge c$  holds for any integers $m,n$ and some absolute constant $c>0$, then we have
    $$ \max_{\chi({\rm mod}\;q)\atop \chi\neq\chi_0}\Big|\sum_{n\le N}f(n)\chi(n)\Big|\ge \sqrt{N}\exp\bigg((\sqrt2+o(1)){\sqrt{\frac{\log (q/N)\log_3(q/N)}{\log_2 (q/N)}}}\bigg).$$ 
\end{thm}
This improves previous result of \cite{DLSZ}, at a cost of assuming an additional condition for $f(\cdot)$.

Recently, Harper \cite{Harper23} showed some moments results for the mixed character sums. He conjectured that, when $N$ goes beyond the size of $\exp(\sqrt {\log q})$, the influence of $f(\cdot)$ will surpass $\chi(\cdot)$, which can lead to a new upper bound for low moments of $\sum_{n\le N}f(n)\chi(n)$. See \cite{Harper23} for more details.
\section{Proof of Theorem \ref{thm1.1}}
The Proof of Theorem \ref{thm1.1} relies on Hough's work \cite{Hou}, which is essentially the resonance method of Soundararajan \cite{Sound}.
Let $Y=q/N$ be large and $\lambda=\sqrt{\log Y\log_2 Y}$. Define the multiplicative function $r$  for any prime $p$:
$$r(p)=\begin{cases}
   \frac{\lambda}{\sqrt p \log p}, &  \lambda\le p\le \exp((\log\lambda)^2),\\
   0, & {\rm otherwise.}
   \end{cases}$$
   Denote
   $$S_\chi(N):=\sum_{n\le N}f(n)\chi(n).$$
    We define the resonator $$R_\chi=\sum_{n\leq Y}r(n)f(n)\chi(n),$$
    and 
   $$M_1(R,q):=\sum_{\chi\neq\chi_0({\rm mod}\;q)}|R_\chi|^2,$$
$$M_2(R,q):=\sum_{\chi\neq\chi_0({\rm mod}\;q)}S_\chi(N)|R_\chi|^2.$$
    then we have 
\[\max_{\chi\neq\chi_0({\rm mod}\;q)}|S_\chi(N)|\geq {\frac{|M_2(R,q)|}{M_1(R,q)}}.\]
 By the orthogonal relation, we have that 
\[\frac{1}{\varphi(q)}\sum_{({\rm mod}\;q)}\chi(n)\ol{\chi}(m)=\begin{cases}
    1&\text{ if }\;q|n-m,\\
    0& \text{otherwise}.
\end{cases}\]
So for $M_1$ we have  
$$M_1(R,q)\le \sum_{\chi({\rm mod}\;q)}|R_\chi|^2\le\varphi(q)\sum_{m\le Y}r(m)^2.$$
For $M_2$ we firstly have
$$|S_{\chi_0}(N)||R_{\chi_0}|^2\le N\Big(\sum_{m\le Y}r(m)\Big)^2\le NY\sum_{m\le Y}r(m)^2=q\sum_{m\le Y}r(m)^2.$$
Thus we have
\begin{align*}
   & M_2(R,q)+O\Big(q\sum_{m\le Y}r(m)^2\Big)\\
   &\ge\sum_{\chi({\rm mod}\;q)}S_\chi(N)|R_\chi|^2\\
   &=\varphi(q)\sum_{m,n\le Y}r(m)r(n)\sum_{k\le N\atop mk=n}f(mk)\overline f(n)\\
   &=\varphi(q)\sum_{k\le N}r(k)\sum_{m\le Y/k}r(m)^2.
\end{align*}
By the proof of Page 105 in \cite{Hou}, we have 
$$\sum_{m\le Y/k}r(m)^2=(1+o(1))\sum_{m\ge1}r(m)^2.$$
Combined with $M_1$ we have
\begin{align*}
\frac{M_2(R,q)}{M_1(R,q)}
\ge(1+o(1))\sum_{k\le N}r(k)+O(1).
\end{align*}
Then we completes the proof by following Page 105--107 of \cite{Hou}.
\section{Proof of Theorem \ref{thm1.2}}
The proof of Theorem \ref{thm1.2} is highly based on the work of La Bret\`eche and Tenenbaum \cite{BT}, which combines sharp evaluation of GCD sums and resonance method. This kind of technique can date back to Aistleitner's work \cite{Ais}.
\begin{lem}\label{lem2.3}

    Let $\M$ be any set of positive  integers with $|\M|=K$. Then as $K\to\infty$, we have
    $$\max_{|\M|=K}\sum_{m,n\in\M}\sqrt{\frac{(m,n)}{[m,n]}}=K\exp\bigg((2\sqrt2+o(1))\sqrt{\frac{\log K\log_3K}{\log_2K}}\bigg).$$
\end{lem}
\begin{proof}
    This is Theorem 1.1 of \cite{BT}.
\end{proof}

    Note that in the proof of the above lemma, the choice for the set $\M$ satisfies $y_\M:=\max_{m\in\M}P(m)\le (\log K)^{1+o(1)}$.

Let $\M$ be a set of positive integers  with $|\M|=\lfloor q/N\rfloor$ such that 
\begin{align*}\frac{1}{|\M|}\sum_{m,n\in\M}\sqrt{\frac{(m,n)}{[m,n]}}=\exp\bigg((2\sqrt2+o(1))\sqrt{\frac{\log(q/N)\log_3(q/N)}{\log_2(q/N)}}\bigg).
\end{align*}
For any $1\le j\le q-1$, let
$$\M_j:=\{m\in\M:\;q|m-j\},$$
$${\mathcal J}:=\{1\le j\le q-1:\;\M_j\neq\varnothing\},$$
and 
$$\M':=\{\min\M_j:\;j\in{\mathcal J}\}.$$
Define for $m\in\M'$
$$r(m’):=\sqrt{|\M_j|},\;\;\;\;\;(m’\in\M_j)$$
$$R_\chi:=\sum_{m’\in\M'}r(m’)f(m’)\chi(m’),$$
and
$$M_1(R,q):=\sum_{\chi\neq\chi_0({\rm mod}\;q)}|R_\chi|^2,$$
$$M_2(R,q):=\sum_{\chi\neq\chi_0({\rm mod}\;q)}|S_\chi(N)|^2|R_\chi|^2.$$
Then 
\begin{align}\max_{\chi\neq\chi_0({\rm mod}\;q)}|S_\chi(N)|\ge\sqrt{\frac{M_2(R,q)}{M_1(R,q)}}.\label{ResMeth}
\end{align}
For $M_1$ we have
\begin{align*}M_1(R,q)&\le M_1(R,q)+|R_{\chi_0}|^2                                                                      \\&=\sum_{\chi({\rm mod}\;q)}\sum_{m’,n’\in\M'}f(m’)\overline f(n’)\chi(m’)\overline\chi(n’)\\
&=\varphi(q)\sum_{m\in\M'}f(m’)\overline f(m’)\\
&=\varphi(q)|\M'|\\
&\le q^2/N.
\end{align*}
For $M_2$, firstly we have
$$|S_{\chi_0}(N)|^2|R_{\chi_0}|^2\le N^2|\M'|^2\le q^2.$$
Then 
\begin{align*}&M_2(R,q)+O(q^2)\\
&=M_2(R,q)+|S_{\chi_0}(N)|^2|R_{\chi_0}|^2\\
&=\sum_{\chi({\rm mod}\;q)}\sum_{m’,n’\in\M'}r(m’)r(n’)\sum_{k,\ell\le N}f(m’k)\overline f(n’\ell)\chi(m’k)\overline\chi(n’\ell)\\
&=\varphi(q)\sum_{m’,n’\in\M'}r(m’)r(n’)\sum_{k,\ell\le N\atop q|m’k-n’\ell}f(m’k)\overline f(n’\ell)\\
&=\varphi(q)\sum_{m’,n’\in\M'}r(m’)r(n’)\sum_{k,\ell\le N\atop m’k=n’\ell}1\\&\;\;\;\;\;\;+\varphi(q)\sum_{m’,n’\in\M}r(m’)r(n’)\sum_{k,\ell\le N\atop q|m’k-n’\ell>0}2\re f(m’k)\overline f(n’\ell)\\
&\gg\varphi(q)\sum_{m’,n’\in\M'}r(m’)r(n’)\sum_{k,\ell\le N\atop q|m’k-n’\ell}1\\
&=\varphi(q)\sum_{k,\ell\le N}\sum_{m’,n’\in\M'\atop q|m’k-n’\ell}r(m’)r(n’).
\end{align*}
For the inner sum with $k,\ell\le N$ fixed, we have
$$\sum_{m’,n’\in\M'\atop q|m’k-n’\ell}r(m’)r(n’)\ge\sum_{m’,n’\in\M'\atop q|m’k-n’\ell}\min\{r(m’)^2,r(n’)^2\}\ge\sum_{m,n\in\M\atop mk=nl}1.$$
So we get
$$M_2(R,q)+O(q^2)\gg\varphi(q)\sum_{k,\ell\le N}\sum_{m,n\in\M\atop mk=nl}1=\varphi(q)\sum_{m,n\in\M}\sum_{k,\ell\le N\atop mk=nl}1.$$
For $m,n\in\M$ fixed, we have  for the inner sum
$$\sum_{k,\ell\le N\atop mk=n\ell}1\ge \frac{N}{\max\{\frac{m}{(m,n)},\frac{n}{(m,n)}\}}\ge\frac{N}{\sqrt{2\frac{m}{(m,n)}\frac{n}{(m,n)}}}=\frac{N}{\sqrt2}\sqrt{\frac{(m,n)}{[m,n]}}.$$
Thus
\begin{align*}
&M_2(R,q)+O(q^2)\\&\gg\varphi(q)N\sum_{m,n\in\M
\atop [m,n]/(m,n)\le N^2/2}\sqrt{\frac{(m,n)}{[m,n]}}\\
&=\varphi(q)N\bigg(\sum_{m,n\in\M
}\sqrt{\frac{(m,n)}{[m,n]}}-\sum_{m,n\in\M
\atop [m,n]/(m,n)> N^2/2}\sqrt{\frac{(m,n)}{[m,n]}}\bigg)\\
&\gg\varphi(q)N|\M|\exp\bigg((2\sqrt2+o(1))\sqrt{\frac{\log(q/N)\log_3(q/N)}{\log_2(q/N)}}\bigg).
\end{align*}
Here in the last step we used
\begin{align*}\sum_{m,n\in\M
\atop [m,n]/(m,n)> N^2/2}\sqrt{\frac{(m,n)}{[m,n]}}&\ll N^{-2\eta}\sum_{m,n\in\M
}\bigg({\frac{(m,n)}{[m,n]}}\bigg)^{\frac12-\eta}\\
&\ll N^{-2\eta}\prod_{p\le y_\M}\bigg(1+\frac{2}{p^{\frac12-\eta}-1}\bigg)\\
&\ll N^{-2\eta}\exp\big( y_\M^{\frac12+\eta}\big)\\
&\ll N^{-2\eta}\exp\big( (\log (q/N))^{\frac12+\eta+o(1)}\big)\\
&\ll \exp\big(-\tfrac23\delta(\log q)^{\frac12+\delta}\big)\exp\big((\log q)^{\frac12+\frac{2}{3}\delta}\big)\\
&\ll1.
\end{align*}
with $\eta=\delta/3$, $y_\M:=\max_{m\in\M} P(m)\le(\log (q/N))^{1+o(1)}$ and $N>\exp((\log q)^{\frac12+\delta}).$
At last we have
$$\frac{M_2(R,q)}{M_1(R,q)}\gg N\exp\bigg((2\sqrt2+o(1))\sqrt{\frac{\log(q/N)\log_3(q/N)}{\log_2(q/N)}}\bigg)+O(N).$$
Back to \eqref{ResMeth}, we finish the proof.
	\section*{Acknowledgements}
	Z. Dong is supported by the Shanghai Magnolia Talent Plan Pujiang Project (Grant No. 24PJD140) and the National
	Natural Science Foundation of China (Grant No. 	1240011770). W. Wang is supported by the National
	Natural Science Foundation of China (Grant No. 12500). H. Zhang is supported by the Fundamental Research Funds for the Central Universities (Grant No. 531118010622), the National
	Natural Science Foundation of China (Grant No. 1240011979) and the Hunan Provincial Natural Science Foundation of China (Grant No. 2024JJ6120).

	\normalem


\begin{thebibliography}{99}
		
			







\bibitem{Ais} Aistleitner, C. {\emph Lower bounds for the maximum of the Riemann zeta function along vertical lines}, {\it Math. Ann.}, {\bf 365} (2016), 473--496.


\bibitem{BT} de la Bret\`{e}che, R.; Tenenbaum, G. {\emph Sommes de G\'{a}l et applications}, {\it Proc. Lond. Math. Soc.}, {\bf 119} (2019), 104--134.

\bibitem{DLSZ} Dong, Z.; Li, Z.; Song, Y.; Zhao, S. {\emph Large values of character sums with multiplicative coefficients}, preprint, arXiv:2508.09750.

\bibitem{GS01} Granville, A.; Soundararajan K. {\emph Large character sums}, {\it J. Amer. Math. Soc.}, {\bf 14} (2001), 365--397.

\bibitem{Harper23} Harper, A. J. \emph{The typical size of character and zeta sums is $o(\sqrt{x})$}, preprint, {arXiv:2301.04390.}
        
 
\bibitem{Hou} Hough, B. {\emph The resonance method for large character sums}, {\it  Mathematika},  {\bf 59},(2013), 87--118



\bibitem{Sound} Soundararajan, K. {\emph Extreme values of zeta and $L$-functions}, {\it Math. Ann.}, {\bf 342} (2008), 67--86. 

        

	\end{thebibliography}
\end{document}